\documentclass[twocolumn]{autart}
\usepackage{natbib}

 \usepackage{parskip} 
 \usepackage[amsthm]{ntheorem}
 \usepackage{mathtools}
\usepackage{thmtools}
\usepackage{amsmath}
\usepackage{amssymb}

\usepackage{enumerate}
\usepackage{graphicx}
 
\usepackage{verbatim}
\usepackage{url}
\usepackage{amsfonts} 
\usepackage{graphicx}
\usepackage{color}
 
\usepackage{epsfig}

 \newcommand{\diag}{\operatorname{diag}}

\newcommand*\diff{\mathop{}\!\mathrm{d}}

 \newtheorem{Theorem}{Theorem}
 \newtheorem{Lemma}{Lemma}

 \newtheorem{Corollary}{Corollary}
  \newtheorem{Definition}{Definition}
 
 \newtheorem{Proposition}[Theorem]{Proposition}
 \newtheorem{Remark} {Remark}
 \newtheorem{Example} {Example}

\newcommand {\R}{\mathbb R}

\newcommand{\be}{\begin{equation}}
\newcommand{\ee}{\end{equation}}

\newcommand{\N}{\mathbb N_+}

\begin{document}

\begin{frontmatter}
\title{On Disturbance Propagation in Vehicular Platoons with Different Communication Ranges}
\thanks[footnoteinfo]{
The work of C. Wu and D. V. Dimarogonas was supported by Knut and Alice Wallenberg Foundation and the Swedish Research Council. 
This work was conducted when C. Wu was at the KTH Royal Institute of Technology.\\
The work of M. Zhang was supported in part by the National Natural Science Foundation of China under Grant 62033005 and Grant 62273270, in part by the Natural Science Foundation of Shaanxi Province under Grant 2023JC-XJ-17, and in part by the Natural Science Foundation of Sichuan Province under Grant 2022NSFSC0923, and in part by CAAI-Huawei MindSpore Open Fund under Grant CAAIXSJLJJ-2022-001A.\\
$*$  Corresponding author. \\
E-mail addresses:  \texttt{chengshuai.wu@gmail.com} (C. Wu), 
\texttt{mengzhang2009@xjtu.edu.cn} (M. Zhang),
\texttt{dimos@kth.se} (D. V. Dimarogonas)
}
 \author[First]{Chengshuai Wu},
 \author[Second]{Meng Zhang$^*$},
 \author[Third]{Dimos V. Dimarogonas}
\address[First] {School of Automation Science and Engineering, Xi'an Jiaotong University, Xi'an 710049, China.
}
\address[Second]{School of Cyber Science and Engineering, Xi'an Jiaotong University, Xi'an 710049, China.}
\address[Third] {School of Electrical Engineering and Computer Science, KTH Royal Institute of Technology and Digital Futures, Stockholm  10044, Sweden.
}

\maketitle

\begin{abstract}
In the control of vehicular platoons, the disturbances acting on one vehicle can propagate and affect other vehicles. If the disturbances do not amplify along the vehicular string, then it is called string stable.  However, it is usually difficult to achieve string stability with a distributed control setting, especially when a constant spacing policy is considered.  
This note considers the string unstable cases and studies disturbance propagation in a nonlinear vehicular platoon consisting of $n+1$ vehicles where the (virtual) leading vehicle provides the reference for a constant spacing policy. Apart from the communications between consecutive vehicles, we also assume that each vehicle can receive information from $r$ neighbors ahead, that is, the vehicular platoon has communication range $r$.
Inspired by existing control protocols, a unified distributed nonlinear control law with a variable $r$ is proposed to facilitate the analysis.
For the maximal overshoot of the inter-vehicular spacing errors, we explicitly show that the effect of disturbances, including the external disturbances acting on each vehicle and the acceleration of the leading vehicle, is scaled by $O(\sqrt{ \left \lceil \frac{n}{r} \right \rceil})$ for a fixed $n$. 
This implies that disturbance propagation can be reduced by increasing communication range. 
Numerical simulation is provided to illustrate the main results. 
\end{abstract}


\end{frontmatter}

\section{Introduction}
Vehicular platooning is an intelligent transportation application which allows a string of vehicles to move in a closely spaced manner. This can decrease the aerodynamic drag and increase the traffic throughput \citep{varaiya1993smart}. Centralized control is impractical when the vehicular platoons are large in length considering the limited wireless communication range and the delay in transmission. Therefore, distributed control protocols are considered (see e.g., \cite{zheng2015stability,herman2017disturbance,peters2014leader}). However, the distributed design leads to the so called disturbance propagation in vehicular platoons. 

Many works consider a safety and performance criterion related to disturbance propagation called \emph{string stability} (see the recent tutorial \citep{feng2019string}). There are various definitions for string stability in existing works, but they all assert similar properties that the effect of disturbances does not amplify along the vehicular string and is independent of the platoon length. Depending on where the disturbances under concern act on, string stability w.r.t. the external input (e.g., acceleration) of the leading vehicle~\citep{ploeg2013lp}, and string stability w.r.t. the external disturbances on each vehicles~\citep{besselink2017string, monteil2019string} can be defined. Depending on which kind of signal norms is utilized to characterize the effect of the disturbances, there are mainly two notions, i.e., $\mathcal{L}_2$ string stability \citep{ploeg2013lp, studli2018vehicular,herman2016scaling} and $\mathcal{L}_\infty$ string stability \citep{ swaroop1996string,besselink2017string, monteil2019string}. Specifically, $\mathcal{L}_2$ string stability can be specified using $H_\infty$ norms of the transfer functions related to the dynamics of each vehicles. The characterization of $\mathcal{L}_\infty$ string stability usually relies on Lyapunov methods \citep{swaroop1996string,besselink2017string}. Recently, \cite{monteil2019string} showed that $\mathcal{L}_\infty$ string stability can be elegantly specified based on contraction theory \citep{sontag_cotraction_tutorial}. This method is extended in~\citep{silva2021string} where a disturbance rejection control design is proposed. 

It is known that string stability of vehicular platoons with constant spacing policy can be achieved if the leading vehicle can broadcast its information to all the other vehicles (see e.g.,~\cite{monteil2019string,besselink2017string}). However, it is not practical when the platoon length is large and the broadcast delays should be considered~\citep{peters2014leader}. 
Another solution to string stability is using time headway policies which rely on absolute velocity information. For example, \cite{knorn2013string} considered a vehicular platoon where each vehicle can communicate with two preceding neighbors and showed it is string stable provided a sufficiently large time headway. As a drawback of time headway policies, undesirable large steady-state inter-vehicular distances can occur \citep{herman2017disturbance}.
 
Despite the nice properties ensured by string stability, there are many negation results suggest that it is usually difficult to achieve string stability especially when the constant spacing policy is considered. For example, \cite{seiler2004disturbance,barooah2005error} showed that vehicular platoons with double integrator dynamics using only relative spacing information and constant spacing policies are always $\mathcal{L}_2$ string unstable for any linear controller.
\cite{inforflow2006} showed that string stability cannot be guaranteed for non-cyclic vehicular platoons with constant spacing policies if the communication range is limited (see also \citep{inforflow2005acc}). Specifically, they proved that string stability can be ensured if at least one vehicle in the platoon has a specified large enough communication range related to the platoon length.
\cite{peters2016cyclic} showed that 
cyclic vehicular platoons with constant spacing policies using only the information of the predecessor are not even stable when the platoon length is large enough. 
Recently, \cite{farnam2019toward} provided an impossibility result for string stability which covers a broad range of control designs. 

For the string unstable cases, analyzing the effect of disturbance propagation is of both theoretical and practical interest. For linear vehicular platoons, this can be analyzed by a transfer function matrix that relates the vector of all external disturbances to the vector of all spacing errors (see e.g., \cite{seiler2004disturbance, inforflow2006}). 
Lyapunov methods for studying disturbance propagation have also been reported (see e.g., \cite{li2019complex, herman2017disturbance}), and the results are essentially on the convergence rate of the platoon to a desired spacing configuration since disturbance propagation can be reduced by increasing the convergence rate. 

Unlike the aforementioned works, the disturbance propagation problem studied in this paper is formulated by the 
relation between the disturbances and the maximal overshoot of the inter-vehicular spacing errors. Note that if this relation is independent of the platoon length, then a $\mathcal{L}_\infty$ string stability-like result is achieved. Similar to \citep{inforflow2006,li2019complex}, we study 
vehicular platoons with variable communication ranges, that is, each vehicle can receive information from certain amount of neighbors ahead. 
Compared to the aforementioned works which focus on linear control protocols,
our main contribution is related to a contraction theory framework which allows to analyze nonlinear vehicular platoons and can provide explicit bound for the maximal overshoot of spacing errors subject to 
external disturbances acting on each vehicle and the acceleration of the leading vehicle. 
For a vehicular platoon consisting of $n+1$ vehicles with communication range $r$, we show that the effect of disturbances on the aforementioned bound is scaled by $O(\sqrt{ \left \lceil \frac{n}{r} \right \rceil})$ for a fixed $n$. 
The analysis is conducted for a proposed unified nonlinear control protocol which allows considering a variable $r$. We believe that the proposed framework can also be applied to other 
platooning protocols, especially nonlinear ones, since it essentially only requires that the platooning systems satisfy certain contractive properties. To demonstrate this, a distributed heterogeneous control design is given in Section~\ref{sec:Lstring}, with which the proposed framework shows that $\mathcal{L}_\infty$ string stability can be achieved.

The remainder of this paper is organized as follows. The next section reviews some basic notions and results in contraction theory. 
Section~\ref{sec:profor} presents the configuration and modeling of the studied  vehicular platoons and a unified control protocol for different communication ranges.
The main result is given in Section~\ref{sec:main}. Section~\ref{sec:Lstring} describes a byproduct of the main result, i.e., a  $\mathcal{L}_\infty$ string stability result. 
Simulation analysis is provided in Section~\ref{sec:sim}.

\section{Preliminaries}

Let $|\cdot|: \R^n \to \R_{\geq 0}$ denote a vector norm. The induced matrix norm  $||\cdot||: \R^{n \times p } \to \R_{\geq 0}$ is $||A|| :=  \max_{|x| = 1} |Ax|$. The  matrix measure $\mu: \R^{n \times n } \to \R $, associated with $|\cdot|$, is  
$
\mu(A):=\lim_{\varepsilon \rightarrow 0^+}
\frac{|| I_n +\varepsilon A||-1}{\varepsilon}.
$
For $p \in \{1, 2, \infty\}$, let $\mu_p$ [$||\cdot||_p$] denote the matrix measure [matrix norm] associated with the $p$-norm $|\cdot|_p$.
For any measurable vector function $\theta : \R_{\geq 0} \to \R^n$ and a vector norm $|\cdot|$, define $||\theta||: = \operatorname{ess.sup}_{t \geq 0} |\theta(t)|$, i.e., the signal $\mathcal{L}_\infty$ norm induced by the vector norm $|\cdot|$. For example, we use $||\theta||_\infty$ to denote the signal $\mathcal{L}_\infty$ norm induced by $|\cdot|_\infty$.
Note that although $||\cdot||$ is used for both the signal $\mathcal{L}_\infty$ norm and the induced matrix norm, the ambiguity is avoided due to the different representations of the arguments, that is, we always use uppercase letters to represent matrices and lowercase letters to represent vector functions.
For a signal $u(t) \in \R^n$ defined on $[0, T)$ and for some $\tau \in [ 0, T)$, define the related truncation signals: 
\begin{align*}
u_\tau(t) = 
\begin{cases}
u(t), & t \in [0, \tau], \\
0,  & t \in  (\tau, T).
\end{cases} \\
u^\tau(t) = 
\begin{cases}
0, & t \in [0, \tau), \\
u(t),  & t \in [\tau, T).
\end{cases}
\end{align*}
Then, for $0< \tau_1 < \tau_2 < T$, a truncation on the time period $ [\tau_1, \tau_2]$ of $u(t)$  is denoted $u_{\tau_2}^{\tau_1}$. 

Next we review some basic results on contractive systems \citep{sontag_cotraction_tutorial}. Consider the nonlinear time-varying system:
\be\label{eq:tpds}
\dot x(t)=f(t,x(t)),
\ee
where~$f: \R_{\geq 0} \times \Omega \to \R^n$ is $C^1$ and~$\Omega$ is a convex subset of~$\R^n$. 
For~$t\geq t_0 \geq 0$ and~$x_0 \in\Omega$,
we assume that~\eqref{eq:tpds} admits a unique solution~$x(t,t_0,x_0)$  for all~$t\geq t_0$, and that
$x(t,t_0,x_0)\in\Omega $ for all~$t\geq t_0$.
From here on we always take~$t_0=0$ and write~$x(t,x_0)$ for~$x(t,0,x_0)$.  

The system~\eqref{eq:tpds} is called contractive if there exist 
a vector norm~$|\cdot|:\R^n\to\R_{\geq 0} $ and~$\eta>0$ such that for any~$a,b\in \Omega$, 
\be\label{eq:ab}
|x(t,a)-x(t,b)|\leq \exp(-\eta t)|a-b|, \text{ for all } t\geq 0. 
\ee
In other words, any two trajectories converge to each other at an exponential rate. Hence, all the trajectories eventually converge to a unique steady-state solution. 

Let~$J(t,x):=\frac{\partial }{\partial x}f(t,x)$ denote the Jacobian of the vector field~$f$ with respect to~$x$. 
A sufficient condition \citep{entrain2011} guaranteeing~\eqref{eq:ab} is that the matrix measure~$\mu(\cdot):\R^{n\times n}\to \R $, associated with~$|\cdot|$, satisfies
\be \label{eq:mmeta}
\mu(J(t,x))\leq -\eta, \text{ for all }  t\geq 0, x\in\Omega. 
\ee

The next result given by \cite{Desoer1972} shows that a contractive system with an additive input is~input-to-state stable (ISS) \cite[Def. 2.1]{sontag1989smooth}. 

\begin{Proposition} \label{prop:iss}
Consider the system~\eqref{eq:tpds} with an input $u(t) \in \R^n $, i.e., 
\be \label{eq:coninput}
\dot x(t)=f(t,x(t)) +u(t) ,
\ee
where $u: \R_{\geq 0} \to \R^n$ is piecewise continuous. Assume that condition \eqref{eq:mmeta} holds with $\Omega = \R^n$, and $f(t, 0) = 0$ for all $t \geq 0$. Then, 
\be \label{eq:isscon}
|x(t,x_0)|\leq \exp(-\eta t)|x_0|+\int_0^t \exp(-\eta(t-s) ) |u(s)| \diff s, 
\ee
for all $t \geq 0$.
That is, there exists a class $\mathcal{KL}$ function $\beta : \R_{\geq 0} \times \R_{\geq 0}  \to \R_{\geq 0}$ and a class $\mathcal{K}$ function $\gamma: \R_{\geq  0} \to \R_{\geq 0} $ such that 
\be \label{eq:iss}
|x(t,x_0)| \leq \beta( | x_0 | , t)  + \gamma( ||u_t||),  \text{ for all } t \geq 0,
\ee
where $\beta(s, t) = \exp(-\eta t) s$ and $\gamma(s) = s
/ \eta$.
\end{Proposition}

\begin{Remark} \label{re:betterbound}
By computing the integral in~\eqref{eq:isscon}, we have
\[
|x(t,x_0)|\leq \exp(-\eta t)|x_0|+ \frac{1 -\exp(-\eta t)}{\eta} ||u_t||,
\]
for all $t \geq 0$. Note that this offers a better bound on $|x(t,x_0)|$ compared to~\eqref{eq:iss}, but the subsequent analysis will exploit the latter for simplicity.   
\end{Remark}

\section{Problem formulation} \label{sec:profor}

Consider the longitudinal formation control problem of a vehicular string consisting of $(n+1)$ vehicles. The position and velocity of the $i$th vehicle is denoted $p_i(t)$ and $v_i(t)$, $i = 0, 1, \dots, n$, respectively. The state of the leading vehicle $(p_0(t), v_0(t))$ provides a reference for the platoon. The dynamics of the $i$th vehicle, $i = 1, \dots, n$, is described as a second order linear system:
\be \label{eq:platoon}
\begin{aligned}
\dot p_i & =  v_i, \\
  m_i \dot v_i & = u_i +  \theta_i,
\end{aligned}
\ee
where $m_i$ is the mass of the $i$th vehicle, $\theta_i(t)$ denotes the external disturbance, and $u_i(t)$ is the control input of the $i$th vehicle. Note that the 
longitudinal dynamics of vehicles is typically modeled as a nonlinear system, and \eqref{eq:platoon} can be obtained via feedback linearization (see e.g. \cite{stankovic2000decentralized}).

The control goal is to maintain certain designed inter-vehicular distances for the platoon, i.e., the constant spacing policy. 
To formalize this, define a new state $x \in \R^n$ whose $i$th entry is
\[
x_i := p_{i-1} - p_i, \quad i = 1, \dots, n,
\]
i.e., the $i$th inter-vehicular distance. Let $e \in \R^n$ be a vector with positive entries, and its $i$th entry $e_i$, $i = 1, \dots, n$, denotes the desired $i$th inter-vehicular distance.  

As shown in \citep{ploeg2013lp,besselink2017string, monteil2019string}, both the external disturbances acting on each vehicle and the acceleration of the leading vehicle, i.e., $\theta_i$'s and $\dot v_0$, affect the inter-vehicular spacing errors $x_i-e_i, i = 1, \dots, n$. This work seeks to analyze the effects of $\theta_i$'s and $\dot v_0$ simultaneously. 
To facilitate this, we exploit a set of decentralized tracking controllers such that the velocity of the $i$th vehicle can track a desired value $v_{di}(t)$ given by
\begin{align}
v_{di}(t) & := d_i(x_i(t), x_{i+1}(t)) + v_0(t),~ i= 1, \dots, n-1, \nonumber \\ 
v_{dn}(t) &:= d_n(x_n(t)) + v_0(t) . \label{eq:defvdi}
\end{align}
The mappings $d_{i}$, $i = 1, \dots, n$, are $C^1$ w.r.t the arguments and satisfy $d_i(e_i, e_{i+1}) = 0$, $i = 1, \dots, n-1$, and $d_n(e_n) = 0$. That is, the desired velocity of the platoon is $v_0(t)$ when the desired configuration $x = e$ is achieved. Throughout the note (except Section~\ref{sec:Lstring}) we consider a set of $d_i$'s such that there exist positive constants $\eta_1 $ and $c$,  
\begin{align}  \label{eq:cntj}
& \frac{\partial d_i}{\partial x_i} > 0, ~
 \frac{\partial d_i }{\partial x_{i+1}} <  0, ~
\frac{\partial d_i}{\partial x_i} + \frac{\partial d_i }{\partial x_{i+1}} \geq \eta_1 , \\
& \max \left \{  \left \vert \frac{\partial d_i}{\partial x_i  } 
\right \vert , 
\left \vert \frac{\partial d_i }{\partial x_{i+1}} \right \vert  \right \} \leq c,  \label{eq:cntj1}
\end{align}
for all $i = 1, \dots, n$, and all $x \in \R^n$. 
Note that $\frac{\partial d_n }{\partial x_{n+1}} = 0$.

\begin{Remark} The mappings $d_i$ can be viewed as a formation protocol \citep{mesbahi2010graph} for the platoon with the first order dynamics $\dot p_i = v_i$, $i = 1, \dots, n$. For example, we can design $d_i$'s as the 
standard linear formation protocol
\be \label{eq:lineard}
d_i(x_i, x_{i+1}) = \ell_i^p(x_i - e_i) - \ell_i^f(x_{i+1} - e_{i+1}),
\ee
where $\ell^p_i, \ell_i^f >0$, $i = 1, \dots, n$. Note that Conditions~\eqref{eq:cntj}-\eqref{eq:cntj1} hold in this case with $\ell_i^p > \ell_i^f $ for all $i$. 
\end{Remark}

We assume that consecutive vehicles can communicate information including relative positions and velocities. Additionally, each vehicle can receive information from $r$ neighboring vehicles ahead. In other words, the platoon has communication range $r \in \{1, \dots, n\}$. For practical applications, a large $r$ can cause time-delays in communication. The current work considers a ideal case without delays to simplify the analysis. We note however that time-delays could be considered in the suggested framework and are indeed a topic of future work.
 
Let $d_i =0$ and $v_i = v_0$ for all $i \leq 0$. We propose the following control protocol with $r \in \{1, \dots, n \}$ 
\be \label{eq:placon}
\begin{split} 
\frac{1}{m_i} u_i & = -k_i \left ( v_i - \sum \limits_{j=0}^{r-1} d_{i-j} - v_{i-r} \right ) + \frac{\partial d_i}{\partial x_i}  (v_{i-1} - v_i) \\  
& + \frac{\partial d_i}{\partial x_{i+1}}  (v_i - v_{i+1}), ~i=1, \dots, n,  
\end{split}
\ee
where $k_i > 0$, $i=1, \dots, n$, are the control gains. Note that the third term on the right side of~\eqref{eq:placon} vanishes for $i = n$.
 We can see from~\eqref{eq:placon} that Condition~\eqref{eq:cntj} is actually related to exploiting asymmetric weights on the feedback of $v_{i-1} - v_i$ and $v_i - v_{i+1}$, which is known to have the ability to enhance robustness \citep{monteil2019string,hao2012robustness}.
 
\begin{Remark}
The control design~\eqref{eq:placon} is inspired by some existing platooning protocols. For example, Eq.~\eqref{eq:placon} with $r = 1$ has the same structure as the control protocol given in \citep[Eq. (2)]{herman2017disturbance}, that is, each local controller only exploits relative information w.r.t. the immediate predecessor and follower. Eq.~\eqref{eq:placon} with $r = n$ is similar to the control design in \citep[Eq.~(6)]{monteil2019string}, with the distinction that the latter utilizes the feedback term on the relative position w.r.t. the leading vehicle, i.e, $p_0 - p_i - \sum_{j = 1}^i e_j$. In \eqref{eq:placon} with $r = n$, this term is replaced by the weighted version $\sum_{j=0}^{n-1} d_{i-j}$ (i.e., $\sum_{j=1}^i d_j$). Note that both terms correspond to a direct coupling from the leading to the $i$th vehicle.  
\end{Remark}

The next section details the main results, which provide an upper bound for the maximal overshoot of the inter-vehicular spacing errors, i.e., $\max_i\{|| x_i - e_i ||\}$. 

\section{Main results} \label{sec:main}

We first show that the closed-loop system of~\eqref{eq:platoon} with the control design~\eqref{eq:placon} can be reformulated as a singularly perturbed systems (see e.g. \cite[Ch. 11]{hkhalil2002}).  

Define a new state $z \in \R^n$ with its $i$th entry as 
\[ 
z_i  :=  v_i - v_{di},
\]
where $v_{di}$ is given in~\eqref{eq:defvdi}.
Let $k_{\min} : = \min_i\{k_i\}$ and write the control gains $k_i$'s as
$ k_i = \bar k_i k_{\min}, ~ i = 1, \dots, n.$
That is, $\bar k_i$ can be viewed as a relative control gain w.r.t. $k_{\min}$, and satisfies $ \bar k_i \geq 1$ for all $i \in \{ 1, \dots, n \}$.
Define a parameter $\epsilon :=  \frac{1}{k_{\min}} $ and let $z_i = 0$ for all $i \leq 0$. 
Then the closed-loop system of~\eqref{eq:platoon} with the control laws~\eqref{eq:placon} is
\be \label{eq:plaxz}
\begin{split}  
\dot x_i & = d_{i-1} - d_i + z_{i-1} - z_i,   \\
\epsilon \dot z_i & =  - \bar k_i  \left ( z_i - z_{i-r}  -  \sum \limits_{j=1}^r d_{i-j} \right) +     \epsilon \left (\frac{\theta_i}{m_i} - \dot{v}_0 \right), \\
\end{split}
\ee
with $i = 1, \dots, n$. For simplicity, we define the overall disturbance signal $w \in \R^n$ with its $i$th entry as
\[
w_i := \frac{\theta_i}{m_i} - \dot{v}_0, ~ i = 1, \dots, n.
\]
This definition is of practical interest since both $\theta_i$ and $\dot v_0$ affect the inter-vehicular spacing errors (see e.g., \citep{ploeg2013lp,besselink2017string, monteil2019string}).

Note that the parameter $\epsilon$ can be made arbitrarily small by exploiting a high gain design. This justifies that the above system is indeed a singularly perturbed system whose stability usually depends on the existence of a small enough $\epsilon$.  
By the standard procedure of singular perturbation, i.e., time scale decomposition, system~\eqref{eq:plaxz} can be decomposed into the so-called \emph{reduced system} and \emph{boundary layer system}. Specifically, let $\epsilon = 0$, then the $z$-system achieves the steady state $ h(x) \in \R^n$ instantly defined by 
$
h_i - h_{i-r}  - \sum \limits_{j=1}^r d_{i-j} = 0, ~i = 1, \dots, n,
$
where $h_i(x)$ denotes the $i$th entry of $h(x)$. For $i \leq 0$, we set $h_i = 0$. A straightforward calculation yields
\be \label{eq:defhi}
h_i = \sum_{j =1}^{i-1} d_j, ~ i = 1, \dots, n. 
\ee
For example, if $n=3$, then $h_1 = 0, h_2 = d_1$, and $h_3 = d_1+d_2$.
Define the error state
$
y:= z - h(x)
$,
and a fast time scale $\tau := \frac{t}{\epsilon}$. 
Then the system \eqref{eq:plaxz} in the $(x,y)$-coordinate is
\begin{align}
\dot x & = f(x) + G y,  \label{eq:xsys}\\
\frac{\diff y}{\diff \tau} & = A y  - \epsilon \frac{\partial h}{\partial x} (f(x) + G y)  + \epsilon w, \label{eq:ysys}
\end{align}
where $f := \begin{bmatrix} -d_1 & -d_2 & \cdots & -d_n \end{bmatrix}^T$.
Let $G_{ij}$ and $A_{ij}$ denote the $ij$th entry of  the $n \times n$ matrices $G$ and $A$, respectively. Then, 
\begin{align}
 G_{ij} &= \begin{cases}
-1, & i= j, \\
1,  & i = j+1,\\
0,& \text{otherwise}.
\end{cases} \\
 A_{ij} &= \begin{cases}
-\bar k_i, & i= j, \\
\bar k_i,  & i = j+r,\\
0,& \text{otherwise}.
\end{cases} \label{eq:defA}
\end{align}

Before the main theorem, we first give the next two instrumental results studying the matrix measures of the Jacobians of the two systems~\eqref{eq:xsys} and~\eqref{eq:ysys}, i.e., contractivity of~\eqref{eq:xsys} and~\eqref{eq:ysys}. 
The next result shows that the systems~\eqref{eq:xsys} is contractive if $y$ is treated as a parameter.

\begin{Lemma} \label{lem:x} 
 Define 
$
J(x) =  \frac{\partial }{\partial x}f(x). 
$
Condition~\eqref{eq:cntj} ensures that $\mu_\infty( J(x)) \leq  -\eta_1$ for all $x \in \R^n$. 
\end{Lemma}

The proof is based on that $
\mu_\infty( J)  = \max_i\{  J_{ii}+ \sum_{j \neq i} |  J_{ij}| \}
$. 

The next result gives a sufficient condition such that the system~\eqref{eq:ysys} is contractive when $x$ is viewed as a parameter. 
Here, we denote $P \leq Q$ for $P, Q \in \R^{n \times n}$ if $P_{ij} \leq Q_{ij}$ for all $i, j \in \{1, \dots, n\}$.
For a symmetric matrix $P \in \R^{n \times n}$, we denote $P \succ [\succeq] 0$ 
if $P$ is positive [semi]definite. 

\begin{Proposition} \label{lem:y} 
Define 
$
\bar J(x) :=  A - \epsilon \frac{\partial }{\partial x}h(x) G. 
$
Given $a \in \R^n$ and decompose it as
\[
a = 
\begin{bmatrix}
a^1 \\
\vdots \\
a^m
\end{bmatrix}, ~ a^i \in \R^r,~ i = 1, \dots, m-1, a^m \in \R^{n - (m-1)r},
\]
where $m := \lceil \frac{n}{r} \rceil$.
Let $D \in \R^{m \times m}$ be a positive definite diagonal matrix. 
Define a vector norm $|\cdot|_*: \R^n \to \R_{\geq 0}$ by
\be \label{eq:normstar}
|a|_* := \left | \begin{bmatrix}
  |a^1|_\infty \\
\vdots \\
  |a^m |_\infty
\end{bmatrix} \right |_{D,2},
\ee
where $|\cdot|_{D,2}: \R^m \to \R_{\geq 0}$ denotes a $D$-weighted 2-norm defined by $|b|_{D,2}:= |D b|_2$ for $b \in \R^m$.
Let $\mu_*$ denote the matrix measure induced by $|\cdot|_*$. 
Conditions~\eqref{eq:cntj} and~\eqref{eq:cntj1} with some $c>0$ ensure that 
if
\be \label{eq:conlem2}
-\bar k_i + 2 \bar \epsilon c (r-1) < 0, ~ i = 1, \dots, n, \text{ for some } \bar \epsilon >0, 
\ee
then there exist $\eta_2 > 0$ and a $D$ such that 
\be \label{eq:lem2}
\mu_*(\bar J(x)) \leq -\eta_2, \text{ for all } \epsilon \in (0, \bar \epsilon], \text{ and all } x \in \R^n.
\ee
\end{Proposition}

\begin{proof}
Consider a $m \times m$ partition of $\bar J(x)$ where the $ij$th block is denoted $\bar J^{ij}(x) \in \R^{n_i \times n_j}$ with $n_i = r$ for $i \in \{1, \dots, m-1\}$, and $n_m = n - r(m -1) $. Define $B(x) \in \R^{m \times m}$, with its $ij$th entry denoted $B_{ij}(x)$,  by
\be \label{eq:defB}
B_{ij}(x) := 
\begin{cases}
 \mu_\infty(\bar J^{ij} (x)), & i = j, \\
 || \bar J^{ij}(x)||_\infty, & i \neq j.
\end{cases}
\ee
Note that $B(x)$ is Metzler.
Let $\mu_{D,2}$ denote the matrix measure induced by $|\cdot|_{D,2}$. 
By virtue of $|\cdot|_{D,2}$ being monotonic, \cite{storm1975} proved that (see also \cite[Thm. 2]{netwok_contractive})
\be \label{eq:musj1}
\mu_*(\bar J(x)) \leq \mu_{D,2}(B(x)), \text{ for all } x \in \R^n.
\ee
Furthermore, Condition~\eqref{eq:cntj1} ensures that $B(x)$ is bounded for any $x \in \R^n$. That is, there exists a constant Metzler matrix $\bar B \in \R^{m \times m}$ such that $B(x) \leq \bar B$ for all $x \in \R^n$. By \cite[Prop. 2]{ofir2022minimum} (see also a similar result \cite[Thm. 3.3]{Jafar2022nonEuc}), we have
\be \label{eq:musj2}
 \mu_{D,2}(B(x)) \leq \mu_{D,2}(\bar B), \text{ for all } x \in \R^n.
\ee
A straightforward calculation shows that the matrix $F(x):= - \epsilon \frac{\partial }{\partial x}h(x) G$ is a lower triangular matrix satisfying 
\be \label{eq:Fxeps}
|F_{ij}(x)| \leq 2 \epsilon c, \text{ for all } i > j. 
\ee
Recall~\eqref{eq:defA}, then the matrices $\bar J(x)$  and $B(x)$ are also lower triangular. The diagonal entries of $\bar J(x)$ are
\[
-\bar k_1, ~\epsilon \frac{\partial d_1}{\partial x_2} - \bar k_2, \dots,  ~\epsilon \frac{\partial d_{n-1}}{\partial x_n} - \bar k_n. 
\]
Note that they are all negative for any $\epsilon >0$ due to Condition~\eqref{eq:cntj}.
Therefore, Condition~\eqref{eq:conlem2} ensures that there exists $c_1>0$ such that if $\epsilon \in (0, \bar \epsilon]$, then
\[
\mu_\infty(\bar J^{ii}(x)) < -c_1, \text{ for all } i \in \{1, \dots, m\}, ~x \in \R^n.
\]
Since $B(x)$ is lower triangular, its eigenvalues are the diagonal entries, i.e., $\mu_\infty(\bar J^{ii}(x))$, $i = 1, \dots, m$.
Therefore, $B(x)$ is Hurwitz for all $x$, and thus there exists a lower triangular and Hurwitz $\bar B$. 
It is well known that a Hurwitz and Metzler matrix is diagonally stable, i.e., there exists a diagonal matrix $D \succ 0$ and $\eta_2 > 0$ such that 
$
D^2 \bar B + \bar B^T D^2 \preceq -2\eta_2 I_m
$.
This implies that 
$
\mu_{D,2}(\bar B) = \mu_2(D \bar B D^{-1}) \leq -\eta_2
$. 
Then~\eqref{eq:lem2} follows from~\eqref{eq:musj1} and~\eqref{eq:musj2}. $~~~~ \square   $
\end{proof}

Note that the parameter $\eta_2$, i.e., the contraction rate of the system~\eqref{eq:ysys}, depends on the communication range $r$ implicitly since the vector norm $|\cdot|_*$ in \eqref{eq:normstar} is related to $r$.
The next example suggests that a larger $\eta_2$ can be achieved by increasing $r$. 

\begin{Example}  \label{exa:1}
Let $k_1 = \cdots = k_n$, i.e., $\bar k_1 = \cdots = \bar k_n = 1$. Consider the extreme case $\epsilon = 0$, then $\bar J(x) = A$ and the matrix $B(x) \in \R^{m \times m}$ in~\eqref{eq:defB} is given by 
\[
B_{ij} := 
\begin{cases}
 -1, & i = j, \\
1 , & i = j+1, \\
0, & \text{otherwise}.
\end{cases}
\]  
The symmetric part of $B$, i.e., $B_{\operatorname{sym}}:= (B+ B^T)/2$, is an~$m\times m $ tridiagonal matrix with~$-1$ on the main diagonal, and~$1/2$ on the sub- and super-diagonals. 
\cite{Yueh2005EIGENVALUESOS} showed that the eigenvalues of such a matrix  
are
$
\lambda_i=-1+\cos \Big(\frac{i \pi}{m+1} \Big), ~i=1,\dots,m
$.
Therefore, 
\be \label{eq:mu2b}
\mu_2(B) = \lambda_{\operatorname{max}}(B_{\operatorname{sym}}) = -1+\cos \left( \frac{\pi} {m+1} \right)<0.
\ee
That is, the matrix $D$ in this case can be selected as $I_m$. Eq.~\eqref{eq:mu2b} implies that $\mu_2(B)$ is increasing w.r.t. $m$ if $\epsilon =0$. 
For a small enough $\epsilon >0$, this property still holds since $B(x)$ depends on $\epsilon$ continuously and $\mu_2(B(x))$ is continuous w.r.t. the entries of $B(x)$.  
Recall that $m := \lceil \frac{n}{r} \rceil$. This shows that a larger communication range $r$ can lead to a smaller $\mu_2(B)$, that is, a larger $\eta_2$.
\end{Example}

\begin{Remark}
Note that Lemma~\ref{lem:x} and Proposition~\ref{lem:y} show that the system~\eqref{eq:xsys}-\eqref{eq:ysys} can be viewed as an interconnection of two contracive systems, or two ISS systems by Proposition~\ref{prop:iss}. 
It is well known that the small gain theorem can be applied in this case. However, it requires that the term $ \frac{\partial h}{\partial x} f(x)$ in \eqref{eq:ysys} is bounded for all $x \in \R^n$. For the incremental version of small gain theorem~\citep{netwok_contractive}, it requires that $ \frac{\partial}{\partial x} \left(\frac{\partial h}{\partial x} f(x) \right )$ is bounded for all $x \in \R^n$. Note that these two boundedness conditions are not required in this current work. 
\end{Remark}

Based on the above results, the main theorem below shows that
the platooning system~\eqref{eq:platoon} with the control protocol~\eqref{eq:placon} is semi-globally stable. The argument is related to finding a small enough $\epsilon$ (i.e., a large enough $k_{\operatorname{min}}$). It is inspired by \cite[Thm. 1]{christofides1996singular}, which is a Lyapunov method based on implicit constructions. In our case, these implicit constructions can be made explicit by contraction theory. 
To facilitate the analysis, define
\[
\begin{split}
||G||_{\infty,*} := \max \limits_{ a \neq  0}\frac{ |G a|_\infty}{ |a|_*}, ~ a \in \R^n.
\end{split}
\]

\begin{Theorem} \label{thm:main}
Consider the $(x,y)$-system~\eqref{eq:xsys}-\eqref{eq:ysys} where the mappings $d_i$ satisfy Conditions~\eqref{eq:cntj}-\eqref{eq:cntj1}, and the parameters $\bar k_i$ satisfy Condition~\eqref{eq:conlem2} for some $\bar \epsilon >0$.
For each pair of positive constants $(\delta, \sigma)$, there exists $\epsilon^* \in (0, \bar \epsilon]$, such that if $$\max\{ |x(0)-e|_\infty, |y(0)|_*, ||w||_*\} \leq \delta$$ and $\epsilon \in (0, \epsilon^*]$, then
\begin{align}
\label{eq:thmx}
|x(t)-e|_\infty
 &\leq \exp(-\eta_1 t) |x(0)-e|_\infty   \nonumber \\
&~+ \frac{\epsilon ||G||_{\infty,*} ||w_t||_*}{\eta_1 \eta_2}   + \sigma, \\
|y(t)|_*  \leq & \exp \left(-\frac{\eta_2 t} {\epsilon} \right)  |y(0)|_* + \frac{  \epsilon ||w_t||_*}{\eta_2} + \sigma,   \label{eq:thmy}
\end{align} 
for all $ t \geq 0$. The parameters $\eta_1$ and $\eta_2$ are specified in Condition~\eqref{eq:cntj} and Proposition~\ref{lem:y}, respectively.
\end{Theorem}

\begin{proof} 

For the given $(\delta, \sigma)$ and  $\bar \epsilon >0$, let $\delta_x > 0$ be a constant such that 
\be
\delta_x > \delta + \frac{\bar \epsilon ||G||_{\infty, *} \delta}{\eta_1 \eta_2} + \sigma.
\ee
Note that $\delta_x > \delta$. By $ |x(0) - e|_\infty \leq \delta $ and the fact that $x(t)$ is continuous, we can define $[0,T)$ with $T>0$ to be the maximal interval such that
\begin{align} \label{eq:maxT}
| x(t) - e |_\infty < \delta_x, \text{ for all } t \in [0, T).  
\end{align}
To show by contradiction that $T = \infty$ for $\epsilon$ sufficiently small, we first suppose $T$ is finite.
Pick $\epsilon \in (0, \bar \epsilon]$, using Propositions~\ref{prop:iss} and~\ref{lem:y} on $t \in [0, T)$ yields 
\[
|y(t)|_* \leq \exp \left(-\frac{\eta_2 t} {\epsilon} \right)  |y(0)|_* + \frac{ ||u_t||_*}{\eta_2} , \text{ for all } t \in [0, T), 
\]
where 
$
u(t) :=  -  \epsilon \frac{\partial h}{\partial x} f(x)   + \epsilon w. 
$
Since $x(t)$ is bounded on $t  \in [0, T)$ and the mappings $h$ and $f$ are $C^1$ w.r.t. $x$, there exists a class $\mathcal{K}$ function $\gamma: \R_{\geq 0} \to \R_{\geq 0}$ such that 
\[
\left |- \epsilon \frac{\partial h}{\partial x} f(x) \right|_* \leq \gamma(\epsilon) , \text{ for all } t  \in [0, T).
\]
Hence, 
$
|y(t)|_* \leq \exp \left(-\frac{\eta_2 t} {\epsilon} \right)  |y(0)|_* + \frac{ \gamma(\epsilon) + \epsilon ||w_t||_*}{\eta_2} $,  
for all $ t \in [0, T)$. This implies that
\be \label{eq:yboundnew}
\Vert y_t \Vert_* \leq  \delta +  \frac{ \epsilon \delta + \gamma(\epsilon)}{ \eta_2}, \text{ for all } t \in [0, T).
\ee
That is, both $x(t)$ and $y(t)$ are bounded on $t \in [0, T)$.

Pick $\rho \in [0, T) $. 
Using Lemma~\ref{lem:x} and Proposition~\ref{prop:iss}, together with the fact $f(e) = 0$, yields
\be \label{eq:long}
\begin{split}
|x(t)-e|_\infty & \leq  \exp(-\eta_1 t) |x(0)-e|_\infty \\
&+ ||G||_{\infty, *} \int_0^\rho \exp(-\eta_1 (\rho -s))  |y(s)|_* \diff s \\
&+ ||G||_{\infty, *} \int_\rho^t \exp(-\eta_1 (t -s))  |y(s)|_* \diff s\\
& \leq   \exp(-\eta_1 t) |x(0) -e|_\infty + \rho ||G||_{\infty, *}  || y_\rho||_*  \\
&+ \frac{||G||_{\infty, *} || y^\rho_t ||_* }{\eta_1}, \text{ for all } t \in [0, T).
\end{split}
\ee
Note that $ || y_\rho||_* \leq  || y_t ||_* $. 
Let $\rho$ be small enough such that 
\[
 \rho   ||G||_{\infty, *}  \left  ( \delta +  \frac{ \bar \epsilon \delta + \gamma( \bar \epsilon)}{ \eta_2} \right ) \leq \frac{\sigma}{2}.
\]
Then, Eqs.~\eqref{eq:yboundnew} and \eqref{eq:long} lead to
\[
\begin{split} \label{eq:long1}
  &|x(t)-e|_\infty \leq \exp(-\eta_1 t) |x(0)-e|_\infty +    \frac{||G||_{\infty,*} ||y_t^\rho||_*}{\eta_1} \\ 
  & + \frac{\sigma}{2} \\
  & \leq  \exp(-\eta_1 t) |x(0) -e|_\infty 
  + \frac{\sigma}{2}   \\
 &+  \frac{||G||_{\infty,*}}{\eta_1}   \bigg( 
\exp \left(-\frac{\eta_2 \rho}{\epsilon} \right)
|y(0)|_*  
+  \frac{ \gamma(\epsilon) + \epsilon ||w_t||_*}{\eta_2}  \bigg) ,
\end{split}
\]
for all $ t \in [0, T)$.
Note that the terms $\exp \left(-\frac{\eta_2 \rho}{\epsilon} \right)$ 
 and $\gamma(\epsilon)$ converge to zero as $\epsilon$ goes to zero. Therefore, there exists a small enough $\epsilon_1 > 0$, such that 
if $\epsilon \leq \min\{ \epsilon_1, \bar \epsilon \}$, then 
\[
|x(t)-e|_\infty
\leq \exp(-\eta_1 t) |x(0)-e|_\infty  +  \frac{\epsilon ||G||_{\infty,*} ||w_t||_* }{\eta_1 \eta_2}   + \sigma, 
\]
 for all $ t \in [0, T)$. 
From the definition of $\delta_x$, we have that $| x(t) -e|_\infty < \delta_x$ for all  $t \in [0, T)$ and 
\[
 \limsup_{t \to T} |x(t)-e|_\infty \leq \delta_x - (1- \exp(-\eta_1 T) ) \delta.
\]
By the continuity of $x(t)$ and the assumption that $T$ is finite, there exists $\Delta T > 0$ such that $|x(t) -e|_\infty < \delta_x$ for all $t \in [0, T + \Delta T)$, and this contradicts that $T$ is the maximal value such that \eqref{eq:maxT} holds. Hence, we have $T = \infty$,  and \eqref{eq:thmx} holds for all $t \geq 0$. 
Finally, pick $\epsilon_2$ such that $\gamma (\epsilon_2)/\eta_2 \leq  \sigma $ and let $\epsilon^* = \min \{\epsilon_1, \epsilon_2, \bar \epsilon \}$, then~\eqref{eq:thmx} and \eqref{eq:thmy} hold for all $t \geq 0$.   $~~~~ \square   $
\end{proof}

Inspired by Remark~\ref{re:betterbound} and using a similar argument as the proof of Theorem~\ref{thm:main}, a counterpart of~\eqref{eq:thmx} can be obtained as
\begin{align*}
|x(t)-e|_\infty
 &\leq \exp(-\eta_1 t) |x(0)-e|_\infty    \\
&~+  b(t) \epsilon ||G||_{\infty,*} ||w_t||_*   + \sigma, \text{ for all } t \geq 0.
\end{align*}
where $b(t) := \frac{ (1 -\exp(-\eta_1 t))(1-\exp(-\eta_2 t))  }{\eta_1 \eta_2}$. Note that this offers a more accurate bound on $|x(t)-e|_\infty$ compared to~\eqref{eq:thmx} especially for a short time period.

\begin{Remark}
The above proof shows that the only hard requirement on $\epsilon$ is $\epsilon \leq \bar \epsilon$, where $\bar \epsilon$ is specified in Proposition~2. This requirement ensures that the $y$-system is contractive, that is, it is related to a stability condition. This equivalently provides a lower bound for the control gains $k_i$ since $\epsilon := \frac{1}{k_{min}}$. Apart from this, a smaller $\epsilon$ leads to smaller inter-vehicular spacing errors. That is, the control performance can be improved but at the cost of larger control effort. 
\end{Remark}

Similar to the formulation of $\mathcal{L}_\infty$ string stability (see Def.~\ref{def:string} in Section~\ref{sec:Lstring}), we seek to derive a relation between $\max_i \{||x_i(t) - e_i||\}$ and $||w||_\infty$ in the next result.
It directly follows from Theorem~\ref{thm:main} and the definition of~$|\cdot|_*$ in~\eqref{eq:normstar}, and it shows that the effect of $||w||_\infty$ on  $\max_i \{||x_i(t) - e_i||\}$ is scaled by $O(\sqrt{ \left \lceil \frac{n}{r} \right \rceil})$ for a fixed $n$. That is, the effect of disturbance propagation can be reduced by increasing the communication range $r$.  

\begin{Corollary} \label{cor:sqrtw}
Consider the $(x,y)$-system~\eqref{eq:xsys}-\eqref{eq:ysys} where the mappings $d_i$ satisfy Conditions~\eqref{eq:cntj}-\eqref{eq:cntj1}, and the parameters $\bar k_i$ satisfy Condition~\eqref{eq:conlem2} for some $\bar \epsilon >0$. For any essentially bounded $w(t)$ and any bounded set $\Omega \subset \R^{2n}$, there exists $\epsilon^* > 0$, such that for any $\epsilon \in (0, \epsilon^*]$ and any initial condition in $\Omega$, we have
\begin{align}
||x_i-e_i|| \leq 
 |x(0)-e|_\infty   
+ &  \epsilon c \sqrt{ \left \lceil \frac{n}{r} \right \rceil }  ||w||_\infty   + \sigma , \nonumber  \\
& i = 1, \dots, n, \label{eq:coro}
\end{align}
for some $c , \sigma>0$. 
\end{Corollary}

\begin{proof}
By Theorem~\ref{thm:main}, a small enough  $\epsilon$ ensures that   
\begin{align*}
||x_i-e_i|| \leq 
 |x(0)-e|_\infty  
+ \frac{ \epsilon ||G||_{\infty, *}  ||w||_*}{\eta_1 \eta_2}   + \sigma ,
\end{align*}
for all $i = 1, \dots, n$. 
That is, the effect of the disturbance $w(t)$ on every $x_i$ is related to the norm~$|\cdot|_*$ defined in~\eqref{eq:normstar}. Let $q_i>0$, $i = 1, \dots, m$, be  the $i$th diagonal entry of $D$, and define $\bar q :=\max_i \{q_i\}$. Then for $a \in \R^n$ with the decomposition as in Proposition~\ref{lem:y}, 
\begin{align*}
|a|_* = \left ( \sum \limits_{i=1}^m \left ( q_i  |a^i|_\infty \right)^2 \right)^{\frac12} 
 \leq \sqrt{ \left \lceil \frac{n}{r} \right \rceil} \bar q |a|_\infty. 
\end{align*}
 This implies that $||w||_* \leq \sqrt{ \left \lceil \frac{n}{r} \right \rceil} \bar q ||w||_\infty$. Therefore, Eq.~\eqref{eq:coro} holds with  $c:= \frac{  \bar q  ||G||_{\infty, *}  }{\eta_1 \eta_2}$.   $~~~~ \square   $
\end{proof}

As shown in Example~\ref{exa:1} and Corollary~\ref{cor:sqrtw}, the parameter $c$ in~\eqref{eq:coro} also depends on~$n$ and~$r$ implicitly. The exact relation between $c$ and $n$ (or $r$) is not the focus of our current work. However, for a fixed $n$, we can always find a maximal value of $c$  such that \eqref{eq:coro} holds for all $r \in \{1, \dots, n\}$.   

Note from \eqref{eq:coro} that a $\mathcal{L}_\infty$ string stability-like result is achieved if $r = n$. 
Inspired by this, the next section proposes a control design ensuring $\mathcal{L}_\infty$ string stability of the platooning system~\eqref{eq:platoon}. Specifically, it is related to an analogue of Theorem~\ref{thm:main} but with $\sigma = 0$.

\section{A $\mathcal{L}_\infty$ string stability result} \label{sec:Lstring}

To show that the proposed framework has the potential to analyze other platooning protocols, this section demonstrates a $\mathcal{L}_\infty$ string stability result by using the following control law 
\be \label{eq:placonv0}
\begin{split}
\frac{1}{m_i} u_i &= - k_i (v_i - d_i - v_0) + \frac{\partial d_i}{\partial x_i} (v_{i-1} - v_i) \\ &~+ \frac{\partial d_i}{\partial x_{i+1}}  (v_i - v_{i+1}), ~i = 1, \dots, n.
\end{split}
\ee
Comparing~\eqref{eq:placonv0} to \eqref{eq:placon} with $r=n$,
the term $\sum_{j=1}^{n-1} d_{i-j}$ (i.e., $\sum_{j=1}^{i-1} d_j$)
is dropped here. As a result, each local controller only depends on the information of the immediate predecessor and follower and the velocity of the leading vehicle $v_0(t)$. Compared to the control design given by~\cite{monteil2019string}, the relative position w.r.t. the leading vehicle, i.e., $p_0(t) - p_i(t)$, $i =1 , \dots, n$, are not required. 

The subsequent analysis show that 
the control law~\eqref{eq:placonv0} ensures $\mathcal{L}_\infty$ string stability of the platoon, which is formally defined  below.

\begin{Definition} \label{def:string}
Consider the platooning system~\eqref{eq:platoon} and define the lumped error state 
\be \label{eq:defxi}
\xi:= \begin{bmatrix}  (x - e)^T &  v_1 - v_0 &  \cdots &  v_n - v_0
\end{bmatrix}^T.
\ee
The system~\eqref{eq:platoon} is called $\mathcal{L}_\infty$ string stable if there exist class $\mathcal{K}$ functions $\alpha_i$, $i = 1,  2$, such that for any initial condition $\xi(0) \in \R^{2n}$, and any essentially bounded $\dot v_0(t) $ and $ \theta(t)$, 
\be \label{eq:defstr}
\Vert x_i - e_i \Vert \leq \alpha_1 ( | \xi(0) |_\infty ) + \alpha_2( \Vert w \Vert_\infty ),  
\ee
 for all $ i \in \{ 1, \dots, n\}$, and all~$ n \in \N$.
\end{Definition}

\begin{Remark} 
If $\theta_i(t) \equiv 0$ for all $i$, then Definition~\ref{def:string} recovers the $\mathcal{L}_\infty$ string stability given in \cite[Def. 1]{ploeg2013lp}. 
Note that Condition~\eqref{eq:defstr} is independent of the length of the platoon.
\end{Remark}

Similar to the derivation of~\eqref{eq:xsys}-\eqref{eq:ysys},
the closed-loop system of~\eqref{eq:platoon} with the control laws~\eqref{eq:placonv0} is obtained as
\begin{align} 
\dot x & = f_1(x) + G z, \label{eq:redu1} \\
 \frac{\diff z }{\diff \tau} & =  - \diag ( \bar k_i ) z + \epsilon w,  \label{eq:bola1}
\end{align}
where
\[
f_1(x)  := \begin{bmatrix}
 - d_1(x_1, x_2) \\ d_1(x_1, x_2) - d_2(x_2, x_3) \\ \vdots \\ d_{n-1}(x_{n-1}, x_n) - d_n(x_n)
\end{bmatrix}.   
\]

\begin{Theorem}
Consider the platooning system~\eqref{eq:platoon} with the distributed controllers in~\eqref{eq:placonv0}, i.e., the system~\eqref{eq:redu1}-\eqref{eq:bola1}. Pick mappings $d_i$ such that Condition~\eqref{eq:cntj1} holds, and there exists $\eta>0$ such that
\begin{align} \label{eq:cntjx1}
& \frac{\partial d_i}{\partial x_i} > 0, ~
 \frac{\partial d_i }{\partial x_{i+1}} <  0,~i = 1, \dots, n  \nonumber \\
&  \frac{\partial d_1}{\partial x_1} + \frac{\partial d_1 }{\partial x_2} \geq \eta, \\
& \frac{\partial d_i}{\partial x_i} + \frac{\partial d_i }{\partial x_{i+1}} \leq \frac{\partial d_{i+1} }{\partial x_{i+1}} + \frac{\partial d_{i+1} }{\partial x_{i+2}} -\eta, ~i = 1, \dots, n-1, \nonumber
\end{align}
for all $x \in \R^n$ and all $i \in \{1, \dots, n\}$, then the platooning system~\eqref{eq:platoon} is $\mathcal{L}_\infty$ string stable. 
\end{Theorem}

\begin{proof} Define~$J_1(x) := \frac{\partial}{\partial x}f_1(x)$. Condition~\eqref{eq:cntjx1} ensures that $\mu_\infty(J_1(x)) \leq -\eta$ for all $x \in \R^n$. 
Additionally, the Jacobian of the $z$-system~\eqref{eq:bola1} is  $-\diag ( \bar k_i )$, and $\mu_\infty(-\diag ( \bar k_i )) = -1$.

Note that the $z$-system~\eqref{eq:bola1} is independent of $x$. That is, the $(x, z)$-system~\eqref{eq:redu1}-\eqref{eq:bola1} is a series connection of two contractive systems.
Recall that such a system is also contractive (see e.g., \cite{ofir2022sufficient}). 
Specifically, this implies that $x(t)$ and $z(t)$ are uniformly bounded if $w(t) \in \mathcal{L}_\infty$. 
Using Proposition~\ref{prop:iss} leads to 
\begin{align*}
&|x(t) - e|_\infty
 \leq \exp(-\eta_1 t) |x(0)|_\infty  +  \frac{||G||_\infty }{\eta } ||z_t||_\infty , \\
&|z(t)|_\infty  \leq \exp \left(-\frac{ t} {\epsilon} \right)  |z(0)|_\infty +   \epsilon ||w_t||_\infty,  \text{ for all } t \geq 0 .
\end{align*} 
This implies that
$
||z||_\infty \leq   |z(0)|_\infty 
+  \epsilon || w||_\infty. 
$
Hence, 
\be \label{eq:xlinfty}
\begin{split}
||x-e||_\infty
& \leq  |x(0) -e|_\infty  +  \frac{||G||_\infty }{\eta } (|z(0)|_\infty 
+   \epsilon  || w||_\infty).
\end{split}
\ee
Note that $ ||x_i -e_i|| \leq ||x-e||_\infty $ for all $i$.
Recall that $z_i := v_i - d_i - v_0$, and $d_i(e) = 0$. By Condition~\eqref{eq:cntj1}, the term $  | x(0) -e|_\infty + \frac{||G||_\infty  |z(0)|_\infty } {\eta} $ can be viewed as a weighted vector norm for $\xi(0)$. That is, there exists a constant $c'$  such that 
\be \label{eq:xinorm}
| x(0) -e|_\infty + \frac{||G||_\infty  |z(0)|_\infty } {\eta} \leq c' |\xi(0)|_\infty.
\ee
Combining~\eqref{eq:xlinfty} and~\eqref{eq:xinorm} proves that Condition~\eqref{eq:defstr} in Definition~\ref{def:string} holds and this completes the proof.   $~~~~ \square   $
\end{proof}

\begin{Remark}
Note that Condition~\eqref{eq:cntjx1} is less easier to satisfy than Condition~\eqref{eq:cntj}. This is due to that the former is related to a heterogeneous control design, that is, letting each vehicle
have different control settings. 
For example, let $d_i$'s be the linear form in~\eqref{eq:lineard}. 
Condition~\eqref{eq:cntjx1} translates to:
$
\ell_1^p  > \ell_1^f, ~~
\ell_i^p - \ell_i^f  < \ell_{i+1}^p - \ell_{i+1}^f 
$,
and it is satisfied if we fix $\ell_i^p$ [$\ell_i^f$] as the same constant for all $i$, and $\ell_i^f$ [$\ell_i^p$] as a decreasing [increasing] sequence w.r.t. $i$. Intuitively speaking, this means that the vehicles at the tail [head] part should have a stronger ``connection'' with the predecessor [follower]. 
Note that this may be unpractical for large vehicular platoons since the control parameters can become extremely large/small.
A similar design is reported by
\cite{noniden2004Khatir} who considered a unidirectional platooning protocol and 
showed that string stability is ensured if the control gains of local PID controllers increase w.r.t. $i$. 
\end{Remark}

\section{Numerical validation} \label{sec:sim}
The results in Section~\ref{sec:main} are illustrated via simulations for a platooning system consisting of $11$ vehicles using the control protocol~\eqref{eq:placon}  with
\[
 d_i(x_i, x_{i+1}) = \ell_i g_i(x_i, x_{i+1}) + b_i (x_i - e_i), ~ i = 1, \dots, 10, \\
\]
where 
$
g_i(x_i, x_{i+1}) = \tanh (\ell_i^p(x_i - e_i) - \ell_i^f(x_{i+1} - e_{i+1}) ). 
$

We consider three cases with communication range $r =1, 3$, and $10$. We use the same settings of control parameters for the three cases for a comparison study. 
The control gains $\ell_i, \ell_i^p$, $\ell_i^f$, and $b_i$ are selected such that~Conditions~\eqref{eq:cntj}-\eqref{eq:cntj1} are satisfied. Specifically, we pick 
$
\ell_i = 0.5, \ell_i^p = 0.18, ~\ell_i^f = 0.18,~b_i = 0.1,
$
for all $i = 1, \dots, 10.$ 
The control gains $k_i$'s in~\eqref{eq:placon} are chosen based on Proposition~\ref{lem:y}, and we set
$
k_1 = \cdots = k_{10} = 5.
$

The simulation results with $r = 1, 3$, and $10$ are illustrated in Fig.~\ref{fig:pla}, where we set $e_1 =\cdots = e_{10} = 10 \operatorname{m}$. To show a convergence result, we pick the disturbances $\omega_i(t)$ as vanishing signals.
The time-varying velocity of the leading vehicle is given as
\[
v_0(t) =  
\begin{cases}
   15, & t \in [0, 5), \\
   5 + 2 t , & t \in [5, 15), \\
   35, & t \in [15, 25), \\
   85-2t, & t \in [25, 35), \\
   15,  & t \in [35, 45), \\
   82.5 - 1.5 t, & t \in [45, 55), \\
   0,  & t \in [55, 65), \\
   -97.5 + 1.5 t , & t \in [65, 75), \\
   15, & t \in [ 75, \infty).
\end{cases}
\]
The external disturbances are 
\[
\frac{\theta_i}{m_i} = c_i \exp(-0.02t) \sin(t) (\text{or } \cos(t)) , ~i= 1, \dots, 10,
\]
where $c_i \in [-3, 3]$.
Fig.~\ref{fig:pla} shows that the bounds for $\max_i \{ ||x_i - e_i|| \}$ are reduced by increasing the communication range $r$. It is also observed that the effect of $\dot{v}_0$  attenuates faster along the string with a larger $r$.

\begin{figure}[t]
\begin{center}
\includegraphics[scale=0.25]{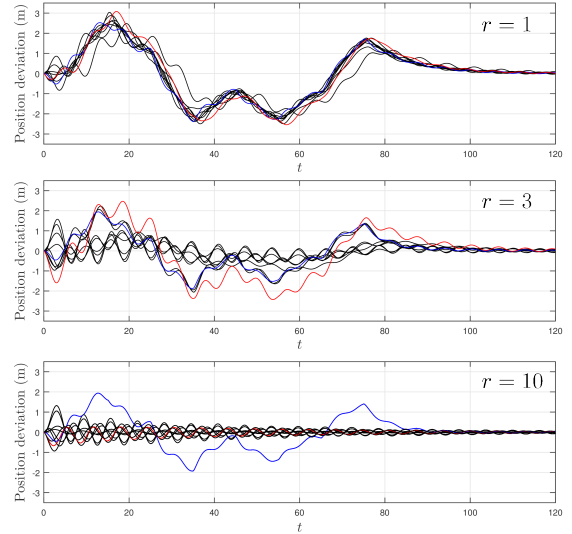}
\caption{Position deviations, i.e., $x_i(t) -e_i$, $i = 1, \dots, 10$, of the platooning systems with communication range $r = 1, 3, 10$. (Blue line: $x_1(t) -e_1$, red line: $x_{10}(t) -e_{10}$.)  }\label{fig:pla}
\end{center}
\end{figure}

\section{Conclusions}
This paper presents a contraction theory framework for analyzing disturbance propagation in nonlinear vehicular platoons with different communication ranges. 
The disturbances under concern include both the external disturbances acting on each vehicle and the acceleration of the leading vehicle. For a proposed unified nonlinear control protocol considering variable communication range $r$ in a vehicular platoon with fixed length $n+1$, we show that the effect of disturbances on the maximal overshoot of spacing errors is scaled by $O(\sqrt{ \left \lceil \frac{n}{r} \right \rceil})$.
Considering that the platoon length can change during merging/splitting real world scenarios, our result may suggest that communicating the information of the platoon length is useful in practical applications, since this allows updating the communication range adaptively to satisfy certain constraints on spacing errors.
We also propose a heterogeneous control law which
only relies on communications between consecutive vehicles and broadcasting the velocity of the leading vehicle such that $\mathcal{L}_\infty$ string stability is achieved.  

\bibliographystyle{abbrvnat}        
\bibliography{platoon}
	
\end{document}